\documentclass[oupthm,crop,info]{Template1}%

\usepackage{graphicx}
\usepackage{multicol,multirow}
\usepackage{amsmath,amssymb,amsfonts}
\usepackage{mathrsfs}
\usepackage{amsthm}
\usepackage{rotating}
\usepackage{appendix}
\usepackage[numbers]{natbib}
\usepackage{cleveref}
\usepackage{mathtools}

\theoremstyle{oupplain}
\newtheorem{theorem}{Theorem}[section]
\newtheorem{lemma}[theorem]{Lemma}
\newtheorem{corollary}[theorem]{Corollary}
\theoremstyle{oupdefinition}
\newtheorem{definition}{Definition}[section]
\theoremstyle{oupremark}
\newtheorem{remark}[theorem]{Remark}

\theoremstyle{oupproof}
\newtheorem{proof}{Proof}
\theoremstyle{oupplain}
\newtheorem{conjecture}[theorem]{Conjecture}
\theoremstyle{oupplain}
\newtheorem{assumption}{Assumption}
\theoremstyle{oupplain}
\newtheorem{proposition}[theorem]{Proposition}

\Crefname{conjecture}{Conjecture}{Conjectures}
\crefname{assumption}{assumption}{Assumptions}

\DeclareMathOperator{\meas}{meas}

\numberwithin{equation}{section}

\articletype{RESEARCH ARTICLE}

\begin{document}

\begin{Frontmatter}

\title[Distribution of sums involving Dirichlet characters over the {$k$}-free integers]{Distribution of sums involving Dirichlet characters over the {$k$}-free integers\thanks{This project is supported in part by the Coordenação de Aperfeiçoamento de Pessoal de Nível Superior - Brasil (CAPES), FAPEMIG Grant Universal No. APQ-00256-23 and CNPq Grant Universal No. 403037/2021-2.}}

\author{Caio Bueno}

\authormark{C. Bueno}

\address{\orgname{Universidade Federal de Minas Gerais, UFMG}.\\
\orgname{Centro Federal de Educação Tecnológica de Minas Gerais, CEFET-MG}.\\
%\orgaddress{\city{Belo Horizonte}, \country{Brazil}}
\email{caiomafiabueno@gmail.com}}



\abstract{Assuming the generalized Riemann hypothesis and a bound for the negative discrete moments of the Riemann zeta function (resp. Dirichlet $L$-functions), we prove the existence of a logarithmic limiting distribution for the normalized partial sums $x^{-1/(2k)}\sum_{n\leq x}f(n)$, where $f$ is either a quadratic Dirichlet character or a modified Dirichlet character, restricted to the $k$-free integers. Moreover, we strengthen a conjecture made by Aymone, Medeiros and the author (cf. Ramanujan J. 59(3):713–728, 2022) concerning the precise order of magnitude for these partial sums.}

\end{Frontmatter}

\section{Introduction and main results}

We start this manuscript by defining a class of functions that are closely related to Dirichlet characters. These functions, that are going to be recurrent in what follows, are called modified Dirichlet characters, since they agree with some character $\chi$ in all, but a finite subset of prime numbers. Precisely, we have the following:
    \begin{definition}
        A function $\Tilde{\chi}:\mathbb{N}\to\{z\in\mathbb{C} : |z|\leq 1\}$ is said to be a modified Dirichlet character if it is completely multiplicative and there exists a Dirichlet character $\chi$ and a finite subset of prime numbers $\mathcal{S}$, such that,
            \begin{align*}
                \Tilde{\chi}(p)
                    \begin{cases}
                        = \chi(p)        &\text{ if $p\notin\mathcal{S}$,}\\
                        \neq\chi(p)     &\text{ if $p\in\mathcal{S}$.}
                    \end{cases}
            \end{align*}
    \end{definition}

We are concerned with the special case where we take $\chi$ to be a real non-principal Dirichlet character of modulus $q$. As in previous papers on the topic, we choose to denote our modified character by $g_\chi$ and define it as
    \begin{align*}
        g_\chi(p)=  \begin{cases}
                        \chi(p) &\text{ if $p\nmid q$,}\\
                        1       &\text{ if $p\mid q$.}
                    \end{cases}
    \end{align*}

That is, our modification set $\mathcal{S}$ consists of primes $p$ that divides the modulus $q$.

Moreover, let $\mu^{(k)}$ be the indicator function of the $k$-free integers.

Recently, \citeauthor{Aymone} proposed the question of finding an example of a multiplicative function $f:\mathbb{N}\to\{-1,0,+1\}$ that resembles\footnote{In the sense that $f$ is supported on the square-free integers and takes values $\pm 1$ at prime numbers.} the Möbius function and such that its partial sums have cancellation greater than square-root. In his investigation, the author proved that
    \[
        \sum_{n\leq x}\mu^{(2)}(n)g_\chi(n)\ll x^{\frac{2}{5}+\varepsilon},\forall\varepsilon>0,
    \]
assuming the Riemann hypothesis.

Later, this result was generalized to the $k$-free integers by Aymone, Medeiros and the author \cite{Aymone-Bueno-Medeiros}. Under the Riemann hypothesis for $L$-functions (or generalized Riemann hypothesis) we have
    \[
        \sum_{n\leq x}\mu^{(k)}(n)g_\chi(n)\ll x^{1/(k+\frac{1}{2})+\varepsilon},\forall\varepsilon>0.
    \]

In 2023, \citeauthor{Liu} made an improvement for the square-free case and obtained, also conditionally to the Riemann hypothesis,
    \[
        \sum_{n\leq x}\mu^{(2)}(n)g_\chi(n)\ll x^{\frac{1}{3}+\varepsilon},\forall\varepsilon>0.
    \]

More recently, the author \cite{Bueno} proved that
    \[
        \sum_{n\leq x}\mu^{(k)}(n)g_\chi(n)\ll x^{\frac{1}{k+1}+\varepsilon},\forall\varepsilon>0,
    \]
conditionally to the generalized Riemann hypothesis.

    \begin{remark}\label{conjecture-for-modifchar}
        The correct order of magnitude for these partial sums is believed \cite[See][]{Aymone-Bueno-Medeiros} to be the same as the conjectural error term in the summatory function of the $k$-free integers, that is, we expect that the summatory function of $\mu^{(k)}g_\chi$ oscillates at most $x^{\frac{1}{2k}+\varepsilon},\forall\varepsilon>0$.
    \end{remark}

    \begin{remark}
        Although we could only find unconditional results for the partial sums of $\mu^{(k)}\chi$ in the literature \cite[See][]{Munsch2014,Liu-Zhang2005}, the same arguments to derive the results above can be made. This will become clear after we show the similarity between the Dirichlet series of $\mu^{(k)}\chi$ and $\mu^{(k)}g_\chi$.
    \end{remark}

For a more thoroughly discussion on this and some related problems, as well as the motivation behind the study of modified Dirichlet characters, we refer to \cite{Bueno}.

The main purpose of this paper is to analyze limiting distributions of Dirichlet characters and modified Dirichlet characters, over the $k$-free integers.
    \begin{definition}\label{def-limiting-distribution}
        Let $h:[0,\infty)\to\mathbb{R}$ and $\nu$ a probability measure on $\mathbb{R}$. We say that the function $h$ has a limiting distribution $\nu$ in $\mathbb{R}$ if
            \begin{align*}
                \lim_{Y\to\infty}\frac{1}{Y}\int_0^Y g(h(y))\,dy=\int_{-\infty}^\infty g(x)\,d\nu(x),
            \end{align*}
        for all bounded, continuous function $g$ on $\mathbb{R}$.
    \end{definition}
    
In the above sense, our first result states:
    \begin{theorem}\label{cor-distribution}
        Assume the generalized Riemann hypothesis. Let $\chi$ modulo $q$ be a primitive real non-principal Dirichlet character and $k\geq 2$ a fixed integer. Additionally, suppose that one of the following holds:
            \begin{enumerate}[(i)]
                \item \label{first-hypothesis} $f=\mu^{(k)}\chi$ (resp. $\mu^{(k)}g_\chi$), $k$ is even and $\sum_{0<\gamma\leq T}|\zeta'(\rho)|^{-2}\ll T^{1+\frac{1}{k}-\varepsilon},\forall\varepsilon>0$;
                \item \label{second-hypothesis} $f=\mu^{(k)}\chi$ (resp. $\mu^{(k)}g_\chi$), $k$ is odd and $\sum_{0<\gamma\leq T}|L'(\rho,\chi)|^{-2}\ll T^{1+\frac{1}{k}-\varepsilon},\forall\varepsilon>0$;
            \end{enumerate}
        where $\rho=\frac{1}{2}+i\gamma$ are the non-trivial zeros of $\zeta(s)$ in the first case or $L(s,\chi)$ in the second case. Then $e^{-\frac{y}{2k}}\sum_{n\leq e^y}f(n)$ has a limiting distribution $\nu_k$ on $\mathbb{R}$.
    \end{theorem}

Under these stronger hypothesis, we also obtain a result of similar magnitude than our expected sharp bound (see \Cref{conjecture-for-modifchar} above).
    \begin{theorem}\label{log-meas-zero-thm}
        Assume the generalized Riemann hypothesis and let $\chi$ modulo $q$ be a primitive real non-principal Dirichlet character and $k\geq 2$ a fixed integer. Furthermore, assume \Cref{first-hypothesis} or \Cref{second-hypothesis} from the previous theorem. Then, we have
            \begin{align}\label{excp-logarithmic-measure-result}
                \sum_{n\leq x}f(n)\ll_{\varepsilon}x^{\frac{1}{2k}}(\log x)^{\frac{1}{2}+\varepsilon},\forall\varepsilon>0,
            \end{align}
        except on a set of finite logarithmic measure\footnote{We say that a set $\mathcal{A}$ has finite logarithmic measure if $\int_{\mathcal{A}\cap [1,\infty)}\,\frac{dx}{x}<\infty$.}.
    \end{theorem}

An active line of research is to study the behaviour of the partial sums of multiplicative functions. A commonality present in some functions that we find in the literature is their connection to the Riemann zeta function.

In most cases, this connection can be readily seen by means of the Euler product formula. A simple calculation shows that the Möbius function can be related to the Riemann zeta function by
    \begin{align*}
        \frac{1}{\zeta(s)}=\sum_{n=1}^\infty\frac{\mu(n)}{n^s},
    \end{align*}
for $\Re(s)>1$. Another example is the generating series of the function $\mu^{(k)}$:
    \begin{align*}
        \frac{\zeta(s)}{\zeta(ks)}=\sum_{n=1}^\infty\frac{\mu^{(k)}(n)}{n^s},
    \end{align*}
for $\Re(s)>1$. Therefore, the partial sums of these functions are also linked to the Riemann zeta function by employing the Mellin transform (simply making a partial summation) or Perron's formula.

The functions that we are interested also have a close relation to the Riemann zeta function and to the Dirichlet $L$-functions, as we are going to see more carefully later.

Probability viewpoints were implemented by many authors in the study of problems in Analytic Number Theory and arithmetic functions, given its usefulness in providing heuristics arguments and establish precise conjectures. Particularly, limiting distribution has been notorious in the investigation of the possible true order of magnitude of some summatory functions.

We emphasize that conjectures for moments of the Riemann zeta function are essential in our proofs. Specifically, as in \Cref{first-hypothesis,second-hypothesis} of \Cref{cor-distribution}, we need estimates for the discrete moments:
    \[
        J_{r}(T)\coloneqq\sum_{0<\gamma \leq T}|\zeta'(\rho)|^{2r},
    \]
where $\rho=\beta+i\gamma$ are the non-trivial zeros of the Riemann zeta function and $r\in\mathbb{R}$.

In 1989, Gonek (for $r\leq 0$) and Hejhal (for $r>0$) independently conjectured the following bound:
    \begin{conjecture}[\citeauthor{Gonek1989,Hejhal1989}]\label{Gonek-Hejhal-conjecture}
        For all $r\in\mathbb{R}$, we have
            \[
                J_{r}(T)\asymp T(\log T)^{(r+1)^2}.
            \]
    \end{conjecture}

Later, a restriction on $r$ arose in the work of Hughes, Keating and O'Connell \cite{Hughes-Keating-OConnell2000}. In fact, we have a very precise conjecture obtained through a model for the Riemann zeta function using the characteristic polynomial of a random unitary matrix.
    \begin{conjecture}[\citeauthor{Hughes-Keating-OConnell2000}]\label{keating-hughes-oconell-conject}
        For all $r\in\mathbb{C}$ such that 
        $\Re(r)>-3/2$, we have
            \begin{align*}
                J_{r}(T)\sim \frac{1}{2\pi}\frac{G^2(r+2)}{G(2r+3)}\alpha(r)T(\log T)^{(r+1)^2},
            \end{align*}
        where $G$ is Barnes' $G$-function defined by
            \[
                G(s+1)=(2\pi)^{\frac{s}{2}}\exp\bigg(-\frac{1}{2}(s^2+\gamma s^2+s)\bigg)\prod_{n=1}^\infty \bigg(1+\frac{s}{n}\bigg)^n\exp\bigg(-s+\frac{s^2}{2n}\bigg),
            \]
        with $\gamma$ being the Euler-Mascheroni constant and $s\in\mathbb{C}$, and
            \[
                \alpha(r)=\prod_p\bigg(1-\frac{1}{p}\bigg)^{r^2}\bigg(\sum_{m= 0}^\infty\bigg(\frac{\Gamma(m+r)}{m!\Gamma(r)}\bigg)^2\frac{1}{p^m}\bigg).
            \]
        The product above ranges over all prime numbers $p$.
    \end{conjecture}

Results related to these moments are very limited. Some examples we have are \cite{Gao-Zhao,Gonek1989,Heap-Li-Zhao2022,Milinovich-Ng2014} and \cite{Milinovich2010}, for lower and upper bounds, respectively. We also refer to the Introduction of \cite{Humphries2013}, where Humphries gave a nice survey on the subject.

It seems reasonable that a similar conjecture should hold for moments of Dirichlet $L$-functions.\footnote{See also \cite{Akbary-Ng-Shahabi2014}, where the authors assume a similar type of conjectural bound.}
    \begin{definition}
        Let $\kappa\geq 1$ be an integer, $\chi$ a primitive Dirichlet character of modulus $q$ and $\rho=\beta+i\gamma$ the non-trivial zeros of the Dirichlet $L$-function associated to the character power $\chi^\kappa$. For all $r\in\mathbb{R}$, define
            \[
                \Tilde{J}_{r}(T)\coloneqq\sum_{0<\gamma\leq T}|L'(\rho,\chi^\kappa)|^{2r}.
            \]
    \end{definition}
    \begin{conjecture}\label{Lfuncwithcte-conj}
        For all $r\in\mathbb{C}$ with $\Re(r)>-3/2$, there is a constant $C_{q,r}$, such that
            \[
                \Tilde{J}_{r}(T)\sim C_{q,r} T(\log T)^{(r+1)^2}.
            \]
    \end{conjecture}

In the majority of this paper we assume a weaker hypothesis, namely \Cref{first-hypothesis,second-hypothesis}. However, in \Cref{large-deviation-section}, we may assume \Cref{Gonek-Hejhal-conjecture} and
    \begin{align}\label{conjecture-L}
        \Tilde{J}_{r}(T)\asymp T(\log T)^{(r+1)^2},
    \end{align}
with $\kappa=1$.

Moreover, our constants will be implicit and therefore we refrain from applying \Cref{keating-hughes-oconell-conject,Lfuncwithcte-conj}.

As we are going to adress the topic of distribution functions, we briefly comment on this subject. A distribution function $F:\mathbb{R}\to[0,1]$ satisfies the following properties:
    \begin{itemize}
        \item $F$ is non-decreasing
        \item $\lim_{x\to-\infty}F(x)=0$ and $\lim_{x\to\infty}F(x)=1$
        \item $F$ is ``càdlàg'', i.e., right-continuous and has a limit on the left of each $x\in\mathbb{R}$.
    \end{itemize}

We also know that, given a distribution function $F$, there exists a probability $\mathbb{P}$ related to $F$, i.e., $F_{\mathbb{P}}(x)\coloneqq \mathbb{P}((-\infty,x])$.

Therefore, by showing that $e^{-\frac{y}{2k}}\sum_{n\leq e^y}f(n)$ has a limiting distribution, one could attempt to extract useful heuristics and informations about the statistical behavior of the partial sums $\sum_{n\leq e^y}f(n)$.

This was the framework of \citeauthor{Ng2004}, followed by \citeauthor{Humphries2013}, \citeauthor{Akbary-Ng-Shahabi2014} and then by \citeauthor{Meng2017}. We will discuss some tools and steps in \Cref{section-discussion-limiting-dist}.

In \Cref{large-deviation-section} we also state large deviations results and conjectures, similar to those obtained by \citeauthor{Meng2017} in his work on the distribution of the $k$-free numbers. These are consequences of the existence of a limiting distribution and the approach taken was primarily studied by \citeauthor{Ng2004}, for the partial sums of the Möbius function.

    \begin{remark}
        Proofs contained in this paper are written with more details only for the function $f=\mu^{(k)}\chi$ and, throughout the text, we will point out the main modifications if we take $f=\mu^{(k)}g_\chi$. Keeping in mind that everything works for $\mu^{(k)}g_\chi$, we expect to keep the exposition more simple by adopting this structure.
    \end{remark}

\subsection{Notation}

If $f(x)=O(g(x))$ (likewise $f(x)\ll g(x)$) we mean that exists a constant $C>0$ such that $|f(x)|\leq C g(x)$ for all sufficiently large $x$. If $f\asymp g$, then $g\ll f \ll g$. We write $f\sim g$ to say that $f$ is asymptotically equivalent to $g$, that is, $\lim_{x\to\infty}\frac{f(x)}{g(x)}=1$. Finally, if $f(x)=o(g(x))$, then $|f(x)|\leq \varepsilon g(x)$, for any $\varepsilon>0$ and all sufficiently large $x$. In particular, $f(x)=o(1)$ if $f(x)$ tends to zero as $x\to\infty$.

If any of the above asymptotic notations has one or more subindexes, it means that the implicit constant depends on those indicated parameters (e.g. $O_\varepsilon, \ll_{\varepsilon,k}$).

Unless stated otherwise, the letter $p$ will always be a prime number and the product written as $\prod_p$ means that it runs over all prime numbers.

Finally, for any real number $x$, the integer part of $x$ is denote by $[x]$.

\section{Discussion of results concerning limiting distributions}\label{section-discussion-limiting-dist}

In 2004, \citeauthor{Ng2004} proved, under the Riemann Hypothesis and the conjecture $J_{-1}(T)\ll T$, that $x^{-\frac{1}{2}}\sum_{n\leq x}\mu(n)$ possess a limiting distribution. With the methods of \citeauthor{Montgomery1980}, Ng also studied the tail of this distribution to show large deviation results and conjectures. As a consequence, the following was formulated:
    \begin{conjecture}
        There exists a constant $B>0$ such that
            \begin{align*}
                \overline{\underline{\lim}}_{x\to\infty}\frac{\sum_{n\leq x}\mu(n)}{x^{\frac{1}{2}}{(\log\log\log x)^{\frac{5}{4}}}}=\pm B.
            \end{align*}
    \end{conjecture}

The above is widely believed to be true and the conjectural constant was greatly improved recently \cite[See][]{ng2025primenumbererrorterms}.

Later, in 2013, \citeauthor{Humphries2013} proved that the sum of a weighted Liouville function\footnote{The Liouville function $\lambda(n)$ is defined as the completely multiplicative function that takes value $-1$ at each prime number.},
    \begin{align*}
        \sum_{n\leq x}\frac{\lambda(n)}{n^\alpha}, 0\leq \alpha <\frac{1}{2},
    \end{align*}
has a limiting distribution, also under the Riemann Hypothesis and $J_{-1}(T)\ll T$.

A year after, \citeauthor{Akbary-Ng-Shahabi2014} proved that, if $\varphi(y)$ is a $B^p$-almost periodic function, that is, for any $\varepsilon>0$ there exists a real-valued trigonometric polynomial
    \begin{align*}
        P_{N_\varepsilon}(y)=\sum_{n=1}^{N_\varepsilon}r_{n,\varepsilon}e^{i\lambda_{n,\varepsilon}y},
    \end{align*}
such that
    \begin{align*}
        \limsup_{Y\to\infty}\frac{1}{Y}\int_0^Y|\varphi(y)-P_{N_\varepsilon}(y)|^p\,dy<\varepsilon^p,
    \end{align*}
then $\varphi(y)$ possess a limiting distribution.

As pointed out by the authors, although unpublished, this result was in fact known since 1930s.

Moreover, they used the above to prove a more general statement that, under additional hypothesis, guarantee the existence of limiting distributions for a wide class of summatory functions. Various distributions already known in the literature were obtained as a corollary of their main result. It was also established distributions for error terms of some summatory functions, for example, the sum of weighted Möbius and Liouville function, and the error term for the sum of Möbius function in arithmetic progressions.

In 2017, \citeauthor{Meng2017} adapted ideas from \cite{Ng2004} and used results of \cite{Akbary-Ng-Shahabi2014} to prove that the error term of $\sum_{n\leq x}\mu^{(k)}(n)$, after a normalization, also has a limiting distribution. More precisely, under the Riemann Hypothesis and $J_{-1}(T)\ll T^{1+\varepsilon}$,
    \begin{align*}
        x^{-\frac{1}{2k}}M_k(x)=x^{-\frac{1}{2k}}\bigg(\sum_{n\leq x}\mu^{(k)}(n)-\frac{x}{\zeta(k)}\bigg),
    \end{align*}
possess a logarithmic limiting distribution.

More recently, \citeauthor{ng2025primenumbererrorterms} employed some tools used by \citeauthor{Meng2017} and improved the main result of \cite{Akbary-Ng-Shahabi2014} (namely, conditions for a function to be $B^2$-almost periodic - see \Cref{distribution-result-Ng} in the next section), weakening some conditions and stating its theorems in a more general form.

We refer to \cite{Akbary-Ng-Shahabi2014} and the references therein for a more complete historical overview on the subject of limiting distributions.

\section{Preliminaries}\label{section-preliminaries}

Let $F(s)$ be the Dirichlet series with coefficients $f=\mu^{(k)}\chi$, where $\chi$ is a non-principal Dirichlet character modulo $q$ and $k\geq 2$ a fixed integer. By the Euler product formula,
    \begin{align*}
        F(s)=\frac{L(s,\chi)}{L(ks,\chi^k)},
    \end{align*}
for $\Re(s)>1$.

We focus mostly on quadratic characters and in this case, for $\Re(s)>1$, we have
    \begin{align*}
        F(s)=
                \begin{dcases}
                    \frac{L(s,\chi)P(ks)}{\zeta(ks)} &\text{ if $k$ is even,}\\
                    \frac{L(s,\chi)}{L(ks,\chi)}    &\text{ if $k$ is odd,}
                \end{dcases}
    \end{align*}
where $P(s)\coloneqq\prod_{p\mid q}(1-\frac{1}{p^s})^{-1}$.

Moreover, when working with modified Dirichlet characters we are implicitly assuming that $\chi$ is real. If we let $G(s)$ be the generating Dirichlet series of $\mu^{(k)}g_\chi$, then we can write
    \begin{align*}
        G(s)=
            \begin{dcases}
                \frac{L(s,\chi)P(s)}{\zeta(ks)} &\text{ if $k$ is even,}\\
                \frac{L(s,\chi)}{L(ks,\chi)}\frac{P(s)}{P(ks)}      &\text{ if $k$ is odd,}
            \end{dcases}
    \end{align*}
where $P(s)$ is the same as defined above and $\Re(s)>1$.

Given this representation, it becomes clear that our functions relates to $\zeta(s)$ and $L(s,\chi)$. We thus discuss two classical results contained in the literature: the first one is regarding the functional equation for these functions and the second one being the analogue, for Dirichlet $L$-functions, of the classical von Mangoldt estimate for the number of non-trivial zeros of the Riemann zeta function.

Recalling the asymmetric functional equation for the Riemann zeta function, expressed with the ratio of Gamma function factor,
    \begin{align*}
        \zeta(s)=\pi^{s-\frac{1}{2}}\frac{\Gamma(\frac{1-s}{2})}{\Gamma(\frac{s}{2})}\zeta(1-s),
    \end{align*}
we also have the following for Dirichlet $L$-functions asssociated to a Dirichlet character $\chi$ of modulus $q$:
        \begin{align}\label{functional-eq-L}
            L(s,\chi)=\frac{\tau(\chi)}{i^\delta \sqrt{q}}\Delta(s)L(1-s,\overline{\chi}),
        \end{align}
    where $\tau(\chi)=\sum_{n=1}^q\chi(n)e^{2\pi in/q}$,
        \begin{align*}
            \delta\coloneqq \delta(\chi)=
                    \begin{cases}
                        0   &\text{ if $\chi(-1)=1$,}\\
                        1   &\text{ if $\chi(-1)=-1$},
                    \end{cases}
        \end{align*}
    and $\Delta(s)$ is the ratio of gamma factors\footnote{Usually the ratio is represented by the letter $\chi$, but we use $\Delta$ to avoid confusion with the Dirichlet characters.}, that is,
        \begin{align*}
            \Delta(s)=
                \begin{cases}
                    \pi^{s-\frac{1}{2}}\Gamma(1-\frac{s}{2})/\Gamma(\frac{s+1}{2})  &\text{ if $\delta=1$,}\\
                        \pi^{s-\frac{1}{2}}\Gamma(\frac{1-s}{2})/\Gamma(\frac{s}{2})    &\text{ if $\delta=0$.}
                \end{cases}
        \end{align*}

A bound that we are going to use recurrently follows from an application of Stirling's formula. One can show that
            \begin{align}\label{gamma-factor-bound}
                |\Delta(\sigma+it)|\asymp |t|^{\frac{1}{2}-\sigma}. 
            \end{align}

From this, we have
    \begin{lemma}[\citeauthor{Montgomery-Vaughan}, Theorems~13.18~and~13.23]\label{Lemma-lowerupper-bound-zeta}
        Assume the Riemann hypothesis. For $|t|\geq 1$, and all $\varepsilon>0$, we have
            \begin{align*}
                \zeta(\sigma+it)\ll_\varepsilon
                    \begin{cases}
                        t^{\frac{1}{2}-\sigma+\varepsilon}  &\text{ if $0<\sigma<\frac{1}{2}$,}\\
                        t^\varepsilon   &\text{ if $\sigma\geq \frac{1}{2}$,}
                    \end{cases}
            \end{align*}
        and also
            \begin{align*}
                \frac{1}{\zeta(\sigma+it)}\ll_\varepsilon
                    \begin{cases}
                        t^{\sigma-\frac{1}{2}+\varepsilon}  &\text{ if $0< \sigma\leq \frac{1}{2}-\frac{1}{\log\log t}$,}\\
                        t^\varepsilon   &\text{ if $\sigma\geq \frac{1}{2}+\frac{1}{\log\log t}$.}
                    \end{cases}
            \end{align*}
    \end{lemma}

Clearly a similar result holds for $L(s,\chi)$.

Now onto the second observation. As usual, let $N(T)$ be the number of non-trivial zeros of the Riemann zeta function in the strip $(0,1)$ and up to imaginary height $T$. Von Mangoldt proved that
    \begin{align*}
        N(T)=\frac{T}{2\pi}\log\frac{T}{2\pi e}+O(\log T).
    \end{align*}

Similar to this, we also have:
\begin{theorem}[\citeauthor{Montgomery-Vaughan}, Corollary~14.7]\label{classical-zeros-bound}
    Let $\chi$ modulo $q$ be a primitive Dirichlet character and $N(T,\chi)$ the number of zeros $\rho=\beta+i\gamma$ of $L(s,\chi)$ in the strip $\beta\in(0,1)$ with $|\gamma|\leq T$. We have
        \[
            N(T,\chi)=\frac{T}{2\pi}\log\bigg(\frac{qT}{2\pi e}\bigg)+O(\log (qT)),
        \]
    for $T\geq 4$.
\end{theorem}

In particular, both $N(T)$ and $N(T,\chi)$ are $\sim C_q T\log T$.

\subsection{Ng's result and the existence of a limiting distribution}

Now we formalize the ideas mentioned in \Cref{section-discussion-limiting-dist} concerning the existence of a limiting distribution for a class of summatory functions and present a proposition that is central in our proof.

Before continuing to these result, we first fix a few definitions and consider some hypothesis, similar to Ng's paper \cite{ng2025primenumbererrorterms}.

We have the following general sequences:
    \begin{enumerate}[(i)]
            \item $\boldsymbol{\lambda}=(\lambda_n)_{n\in\mathbb{N}}$ is a non-decreasing sequence of positive numbers.
            \item $\boldsymbol{r}=(r_n)_{n\in\mathbb{N}}$ is a sequence of complex numbers.
    \end{enumerate}

We also define $N_{\boldsymbol{\lambda}}(T)$ as the number of $n\in\mathbb{N}$ such that $\lambda_n\leq T$.

\begin{assumption}\label{assumption-1}
    There exists a constant $\theta\in(0,2)$ such that
    \begin{align*}
        \sum_{0<\lambda_n\leq T}\lambda_n^2|r_n|^2\ll T^\theta.
    \end{align*} 
\end{assumption}

\begin{assumption}\label{assumption-2}
    There exists a positive constants $C$ such that,
        \begin{align*}
            N_{\boldsymbol{\lambda}}(T)\leq CT\log T,
        \end{align*}
    for $T\geq 1$.
\end{assumption}

\begin{assumption}\label{assumption-3}
    If $T\geq 3$,
        \begin{align*}
            \sum_{T<\lambda_n\leq T+1}1\ll \log T.
        \end{align*}
\end{assumption}

We are going to need the following proposition which is a generalization of Lemma 5 of \citeauthor{Meng2017}:
    \begin{proposition}[\citeauthor{ng2025primenumbererrorterms}, Proposition~1.16]\label{Ng2025}
        Let $\boldsymbol{\lambda}$ and $\boldsymbol{r}$ be sequences as above, satisfying \Cref{assumption-1} for some $\theta\in(0,2)$. Also suppose that the sequence $\boldsymbol{\lambda}$ satisfies \Cref{assumption-3}. If $V$ is a real number, $1\leq T<X$ and $\varepsilon>0$ sufficiently small, then
            \begin{align*}
                \int_V^{V+1}\bigg|\sum_{T<\lambda_n\leq X}r_ne^{iy\lambda_n}\bigg|^2\,dy \ll \frac{1}{T^{2-\theta-\varepsilon}}.
            \end{align*}
    \end{proposition}

The succeding is essential in our main theorem.
    \begin{theorem}[\citeauthor{ng2025primenumbererrorterms}, Theorem~1.17]\label{distribution-result-Ng}
        Let $\varphi:[0,\infty)\to\mathbb{R}$ and $y_0$ a non-negative constant such that $\varphi$ is square-integrable in $[0,y_0]$. Suppose that we also have the following: exists a real constant $\kappa$ and sequences $\boldsymbol{\lambda},\boldsymbol{r}$ satisfying \Cref{assumption-1,assumption-2,assumption-3}, such that
            \begin{align*}
                \varphi(y)=\kappa+2\Re\bigg(\sum_{0<\lambda_n\leq X}r_ne^{i\lambda_ny}\bigg)+\mathcal{E}(y,X),
            \end{align*}
        for all $y\geq y_0$, $X\geq X_0>0$, and such that
            \begin{align*}
                \lim_{Y\to\infty}\frac{1}{Y}\int_{y_0}^Y|\mathcal{E}(y,e^Y)|^2\,dy=0.
            \end{align*}
        Then $\varphi(y)$ is a $B^2$-almost periodic function.
    \end{theorem}
    \begin{corollary}[\citeauthor{Akbary-Ng-Shahabi2014}, Theorem~2.9]
        The function $\varphi(y)$ has a limiting distribution.
    \end{corollary}

\subsection{\texorpdfstring{Further discussion about $P(s)$}{}}\label{subsection-discussion-P(s)}

Here we are going to examine more thoroughly the finite product
    \[
        P(s)\coloneqq\prod_{p\mid q}\bigg(1-\frac{1}{p^s}\bigg)^{-1},
    \]
defined above, that depends on the fixed integer $q\geq 3$.

Note that, for each $p\mid q$, this function has simple poles at the pure imaginary numbers:
    \begin{align}\label{zeros-P}
        z_p(t)=i\frac{2\pi t}{\log p}, t\in\mathbb{Z}\backslash\{0\},
    \end{align}
and a pole of order $\omega(q)$ at $s=0$, where $\omega(q)$ is the number of distinct prime factors of $q$.

However, in the half-plane $\{s\in\mathbb{C} : \Re(s)>0\}$, we see that $P(s)$ is analytic. Moreover, if we fix any $\sigma_0>0$, then
    \begin{align}\label{bound-for-P(s)}
        P(s)=\prod_{p\mid q}\bigg(1-\frac{1}{p^s}\bigg)^{-1}\ll \prod_{p\mid q}\bigg(1-\frac{1}{p^{\sigma_0}}\bigg)^{-1}\ll\prod_{p\mid q}p^{\sigma_0}\ll q^{\sigma_0}\ll 1,
    \end{align}
for all $\Re(s)\geq \sigma_0>0$. This fact will be used in the proofs of upcoming lemmas.

Further, as we saw, the generating series of $\mu^{(k)}\chi$ and $\mu^{(k)}g_\chi$ are very similar: for even $k$, the former has the factor $P(ks)$, while the latter has $P(s)$. Thus, by the above considerations, we remark that our proofs should hold similarly for both functions.

If $k$ is an odd positive integer, recall that we have the extra factor $P(s)/P(ks)$ in the series involving the modified characters, whereas its counterpart is represented only by a product of $L$-functions. As we are going to see, this difference between them should not impose any extra difficulties.

The basic reason for this last assertion is that we aim to avoid the poles of $P(s)$ in our computations, since contributions coming from the residues at those points are too large when the modulus $q$ is not a prime power. Therefore, we are mostly going to work in the half-plane $\Re(s)>0$.

\section{Main lemmas}\label{lemmata-section}

The first lemma we are going to prove is for any primitive Dirichlet character. We have the following:
    \begin{lemma}\label{main-lemma}
        Assume the generalized Riemann hypothesis and $\Tilde{J}_{-1}(T)\ll T^{1+\frac{1}{k}-\varepsilon}$. Let $k\geq 2$ be a fixed integer and $\chi$ modulo $q$ a primitive non-principal Dirichlet character. Then, there is an $\varepsilon>0$ such that
            \begin{align*}
                \int_{Z}^{1+Z}\bigg|\sum_{T<\gamma\leq X}\frac{L(\frac{\rho}{k},\chi)e^{iy\gamma/k}}{\rho L'(\rho,\chi^k)}\bigg|^2\, dy\ll_{k,\varepsilon} \frac{1}{T^{\varepsilon}},
            \end{align*}
        for $Z>0$ and $T<X$. Here, $\rho=\frac{1}{2}+i\gamma$ are the non-trivial zeros of $L(s,\chi^k)$.
    \end{lemma}
    \begin{corollary}\label{main-lemma-corollary}
        Let the character $\chi$ modulo $q$ be quadratic. Then the term $1/L'(\rho,\chi^k)$ in the integrand of the previous lemma can be replaced by any of the following functions:
            \begin{multicols}{2}
                \begin{enumerate}[1.]
                    \item $P(\rho)/\zeta'(\rho)$
                    \item $1/L'(\rho,\chi)$
                    \item $P(\frac{\rho}{k})/\zeta'(\rho)$
                    \item $P(\frac{\rho}{k})/(L'(\rho,\chi)P(\rho))$
                \end{enumerate}
            \end{multicols}
        And here $\rho$ is either a non-trivial zero of $\zeta(s)$ or $L(s,\chi)$.
    \end{corollary}

We prove the preceding lemma using \Cref{Ng2025}.
    \begin{proof}[Proof of \Cref{main-lemma}]
        Let $\rho_n=\frac{1}{2}+i\gamma_n$ be the non-trivial zeros of $L(s,\chi^k)$. Taking
            \begin{align*}
                r_n=\frac{L(\frac{\rho_n}{k},\chi)}{\rho_n L'(\rho_n,\chi^k)}\quad\text{ and }\quad \lambda_n=\frac{\gamma_n}{k}
            \end{align*}
        in \Cref{Ng2025}, we see that
            \begin{align*}
                \sum_{0<\lambda_n\leq T}\lambda_n^2|r_n|^2&=\sum_{0<\gamma_n\leq T}\bigg(\frac{\gamma_n}{k}\bigg)^2\bigg|\frac{L(\frac{\rho_n}{k},\chi)}{\rho_n L'(\rho_n,\chi^k)}\bigg|^2\ll_k\sum_{0<\gamma_n\leq T}\frac{|\gamma_n|^{1-\frac{1}{k}+\varepsilon'}}{|L'(\rho_n,\chi^k)|^2}\\
                &\ll_k T^{1-\frac{1}{k}+\varepsilon'}\Tilde{J}_{-1}(T)\ll_k T^{2+\varepsilon'-\varepsilon},
            \end{align*}
        for any $\varepsilon,\varepsilon'>0$, and in the first inequality we used the gamma factor bound \eqref{gamma-factor-bound} and the Lindelöf Hypothesis type of bound for $L$-functions. If $\varepsilon>\varepsilon'$, then $2+\varepsilon'-\varepsilon<2$.

        By \Cref{classical-zeros-bound} we also have that
            \[
                \sum_{T<\gamma_n\leq T+1}1\ll \log T.
            \]

        Therefore, we can take $\theta(\varepsilon)=2+\varepsilon'-\varepsilon$ to obtain
            \[
                \int_{Z}^{1+Z}\bigg|\sum_{T<\gamma_n\leq X}\frac{L(\frac{\rho_n}{k},\chi)e^{iy\gamma_n/k}}{\rho_n L'(\rho_n,\chi^k)}\bigg|^2\,dy\ll_{k,\varepsilon} \frac{1}{T^{\varepsilon}}.
            \]
    \end{proof}

An important tool for us is Perron's formula and the version we are going to apply is from \citeauthor{Montgomery-Vaughan}.
    \begin{theorem}[Perron's Formula]\label{Perron}
        Let $s=\sigma+it$ and $A(s)=\sum_{n=1}^\infty \frac{a_n}{n^s}$. If $x>0$ and $\sigma_0>\max\{0,\sigma_a\}$, where $\sigma_a$ is the abscissa of absolute convergence of $A(s)$, then
            \begin{align*}
                \sum_{n\leq x}a_n=\frac{1}{2\pi i}\int_{\sigma_0-iT}^{\sigma_0+iT}A(s)\frac{x^s}{s}\,ds+R(x,T),
            \end{align*}
        where
            \begin{align*}
                R(x,T)\ll \sum_{\substack{x/2<n<2x \\ n\neq x}}|a_n|\min\bigg\{1,\frac{x}{T|x-n|}\bigg\}+\frac{x^{\sigma_0}+4^{\sigma_0}}{T}\sum_{n=1}^\infty\frac{|a_n|}{n^{\sigma_0}}.
            \end{align*}
    \end{theorem}
    \begin{corollary}\label{error-term}
        Let $\chi$ modulo $q$ be a primitive non-principal Dirichlet character, $x>0$, $T\geq 1$ and $\sigma_0=1+\frac{1}{\log x}$. If $f=\mu^{(k)}\chi$, $k\geq 2$ an integer, we have that
            \begin{align*}
                \sum_{n\leq x}f(n)=\frac{1}{2\pi i}\int_{\sigma_0-iT}^{\sigma_0+iT}\frac{L(s,\chi)}{L(ks,\chi^k)}\frac{x^s}{s}\,ds+R(x,T),
            \end{align*}
        where
            \begin{align*}
                R(x,T)\ll 1+\frac{x\log x}{T}.
            \end{align*}
    \end{corollary}
    \begin{proof}
        Taking $a_n=f(n)$, the error term in \nameref{Perron} is
            \begin{align*}
                R(x,T)\ll\sum_{\substack{x/2<n<2x \\ n\neq x}}\min\bigg\{1,\frac{x}{T|x-n|}\bigg\}+\frac{x}{T}\zeta(\sigma_0).
            \end{align*}
        We estimate the sum above as follows: if $n$ is nearest to $x$, we choose the first member of the minimum. For all other $n$, we choose the second member of the minimum. Thus, the above is
            \begin{align*}
                \ll 1+\frac{x}{T}\sum_{1\leq n\leq x}\frac{1}{n}+\frac{x}{T}\zeta(\sigma_0).
            \end{align*}
        We also note that, since $\sigma_0=1+\frac{1}{\log x}$, $\zeta(\sigma_0)\leq \frac{2}{1-\sigma_0}=2\log x$. Therefore,
            \begin{align*}
                R(x,T)&\ll 1+\frac{x\log x}{T}+\frac{x\log x}{T}\\
                &\ll 1+\frac{x\log x}{T}.
            \end{align*}
    \end{proof}

    \begin{remark}
        It is clear that we have the same error term as above if we take $f=\mu^{(k)}g_\chi$, for a quadratic character $\chi$.
    \end{remark}

Before we proceed further, the next technical lemma will be needed.
    \begin{lemma}\label{sequence-lemma}
        Assume the generalized Riemann hypothesis and let $\chi$ be a primitive real non-principal Dirichlet character modulo $q$. There exists a sequence $\mathcal{T}=(T_n)_{n\in\mathbb{N}}$ such that $n\leq T_n\leq n+1$ and we have
            \begin{align*}
                \frac{1}{\zeta(\sigma+iT_n)}\ll_\varepsilon T_n^{\varepsilon}\quad\text{ and }\quad
                \frac{1}{L(\sigma+iT_n,\chi)}\ll_\varepsilon T_n^\varepsilon,
            \end{align*}
        for any $\varepsilon>0$ and $\sigma\in[-1,2]$.
    \end{lemma}
    \begin{proof}
        The bound $\frac{1}{\zeta(\sigma+iT_n)}\ll T_n^{\varepsilon}$ is proved in \cite{Ng2004} \cite[Also see][Theorem~13.22]{Montgomery-Vaughan}. The idea is to show that the result is valid for the interval $[1/2,2]$ and then use the functional equation to obtain the same bound for $\sigma\in [-1,1/2)$. Since the proof is short, we elucidate below for $\frac{1}{L(\sigma+iT_n,\chi)}$.
        
        First we note that, for $\sigma \in [1/2,2]$, the result follows from Theorem 5.19 of \cite{Iwaniec-kowalski}.
        
        Now, if $\sigma\in[-1,1/2)$, we can write
            \begin{align*}
                L(s,\chi)=\frac{\tau(\chi)}{i^\delta \sqrt{q}}\Delta(s)L(1-s,\overline{\chi}),
            \end{align*}
        where the Gauss sum $\tau(\chi)$, the gamma factor $\Delta(s)$ and the function $\delta(\chi)$ were defined previously in \Cref{functional-eq-L}.

        We have that the Gauss sum above is $|\tau(\chi)|=q^{\frac{1}{2}}$, for any primitive $\chi$ of modulus $q$. Therefore, by \cref{gamma-factor-bound}, it follows that
            \begin{align*}
                \bigg|\frac{1}{L(s,\chi)}\bigg|=\bigg|\frac{1}{\Delta(s)L(1-s,\overline{\chi})}\bigg|\ll T_n^{\sigma-\frac{1}{2}+\varepsilon}\ll_\varepsilon T_n^\varepsilon,
            \end{align*}
        for any $\varepsilon>0$ and $\sigma\in[-1,1/2)$.
    \end{proof}

    \begin{remark}
        As a consequence of this proof we get the sharper bound $\ll T_n^{\sigma-\frac{1}{2}+\varepsilon}$ if $\sigma\in[-1,1/2)$.
    \end{remark}

Combining the two previous results with the Residue Theorem and Cauchy's Integral Formula Theorem, we show an explicit representation for the partial sums of our functions of interest.    
    \begin{lemma}\label{lemma-specific-T}
        Assume the generalized Riemann hypothesis. Let $\chi$ modulo $q$ be a primitive real non-principal Dirichlet character and suppose that all zeros of $\zeta(ks)$ and $L(ks,\chi)$ are simple. Let $f$ be either equal to $\mu^{(k)}\chi$ or $\mu^{(k)}g_\chi$, $k\geq 2$ be a fixed integer and $T\in\mathcal{T}$. Then,
            \begin{align*}
                \sum_{n\leq x}f(n)=\sum_{|\gamma|<T}\frac{L(\frac{\rho}{k},\chi)}{\rho}Z_f(\rho)x^{\rho/k}+E(x,T),
            \end{align*}
        where
            \begin{align}\label{def-Zf}
                Z_f(s)=
                    \begin{dcases}
                        \frac{P(s)}{\zeta'(s)}    &\text{ if $k$ is even and $f=\mu^{(k)}\chi$,}\\
                        \frac{1}{L'(s,\chi)}  &\text{ if $k$ is odd and $f=\mu^{(k)}\chi$,}\\
                        \frac{P(s/k)}{\zeta'(s)}    &\text{ if $k$ is even and $f=\mu^{(k)}g_\chi$,}\\
                        \frac{P(s/k)}{L'(s,\chi)P(s)}  &\text{ if $k$ is odd and $f=\mu^{(k)}g_\chi$},
                    \end{dcases}
            \end{align}
        and, for both even and odd $k$,
            \begin{align*}
                E(x,T)\ll_\varepsilon 1+\frac{x\log x}{T}+\frac{x}{T^{1-\varepsilon}\log x}+x^{\varepsilon}T^{\varepsilon},
            \end{align*}
        for any arbitrarily small $\varepsilon\in (0,\frac{1}{2k})$.
    \end{lemma}
    \begin{proof}
        Suppose first that $k$ is even and $f=\mu^{(k)}\chi$. By \Cref{error-term}, if $\sigma_0=1+\frac{1}{\log x}$,
            \begin{align*}
                \sum_{n\leq x}f(n)=\frac{1}{2\pi i}\int_{\sigma_0-iT}^{\sigma_0+iT}\frac{L(s,\chi)P(ks)}{\zeta(ks)}\frac{x^s}{s}\,ds+O\bigg(1+\frac{x\log x}{T}\bigg).
            \end{align*}
        
        Fix a small $\sigma_1$ such that $\frac{1}{2k}>\sigma_1>0$ and let $\mathcal{C}$ be the boundary of the rectangle with vertices being the points $\sigma_0\pm iT$ and $\sigma_1\pm iT$. It follows that
            \begin{align*}
                \frac{1}{2\pi i}\int_{\sigma_0-iT}^{\sigma_0+iT}&\frac{L(s,\chi)P(ks)}{\zeta(ks)}\frac{x^s}{s}\,ds=\frac{1}{2\pi i}\oint_\mathcal{C}\frac{L(s,\chi)P(ks)}{\zeta(ks)}\frac{x^s}{s}\,ds\\
                &-\frac{1}{2\pi i}\bigg(\int_{\sigma_0+iT}^{\sigma_1+iT}+\int_{\sigma_1+iT}^{\sigma_1-iT}+\int_{\sigma_1-iT}^{\sigma_0-iT}\bigg)\frac{L(s,\chi)P(ks)}{\zeta(ks)}\frac{x^s}{s}\,ds.
            \end{align*}

        Noticing that we are avoiding the poles of $P(ks)$ (see discussion in \Cref{subsection-discussion-P(s)}), by the simplicity of zeros and the Residue Theorem,
            \begin{align*}
                \frac{1}{2\pi i}\oint_\mathcal{C}\frac{L(s,\chi)P(ks)}{\zeta(ks)}\frac{x^s}{s}\,ds=\sum_{|\gamma|<T}\frac{L(\frac{\rho}{k},\chi)P(\rho)x^{\rho/k}}{\rho \zeta'(\rho)}.
            \end{align*}

        We now give an estimate for both integrals over the horizontal line segments. Since $\Re(s)>0$, by \Cref{sequence-lemma} and \Cref{bound-for-P(s)}, we have
            \begin{align*}
                \bigg|\int_{\sigma_1+iT}^{\sigma_0+iT}\frac{L(s,\chi)P(ks)}{\zeta(ks)}\frac{x^s}{s}\,ds\bigg|&\ll_\varepsilon \frac{1}{T^{\frac{1}{2}-2\varepsilon}}\int_{\sigma_1}^{\frac{1}{2}}\bigg(\frac{x}{T}\bigg)^\sigma\,d\sigma+\frac{1}{T^{1-2\varepsilon}}\int_{\frac{1}{2}}^{\sigma_0}x^\sigma\,d\sigma\\
                &\ll\frac{x^{\frac{1}{2}}}{T^{\frac{1}{2}-2\varepsilon+\sigma_1}}+\frac{x}{T^{1-2\varepsilon}\log x}.
            \end{align*}

        For the vertical line segment, we use \Cref{Lemma-lowerupper-bound-zeta} to bound the zeta function in the denominator and obtain
            \begin{align*}
                \bigg|\int_{\sigma_1-iT}^{\sigma_1+iT}\frac{L(s,\chi)P(ks)}{\zeta(ks)}\frac{x^s}{s}\,ds\bigg|&\ll_\varepsilon x^{\sigma_1} \int_{1}^T \frac{t^{\frac{1}{2}-\sigma_1+\varepsilon}t^{k\sigma_1-\frac{1}{2}+\varepsilon}}{t}\,dt\\
                &\ll_\varepsilon x^{\sigma_1}\int_{1}^T t^{\sigma_1(k-1)+2\varepsilon-1}\,dt\ll_\varepsilon x^{\sigma_1} T^{\sigma_1(k-1)+2\varepsilon}.
            \end{align*}

        If $\varepsilon>0$ is small enough, we choose $\sigma_1=\varepsilon$. Thus, we have that
            \begin{align*}
                \sum_{n\leq x}f(n)&=\sum_{|\gamma|<T}\frac{L(\frac{\rho}{k},\chi)P(\rho)x^{\rho/k}}{\rho\zeta'(\rho)}+O\bigg(1+\frac{x\log x}{T}\bigg)\\
                &+O_\varepsilon\bigg(\frac{x}{T^{1-\varepsilon}\log x}\bigg)+O_\varepsilon(x^\varepsilon T^{\varepsilon}).
            \end{align*}

        Observe that the case when $k$ is odd differs very slightly from the above: there is an absence of the factor $P(ks)$ and, instead of $\zeta(ks)$, we have $L(ks,\chi)$. Recalling that \Cref{sequence-lemma} gives the same bound for both $1/\zeta(ks)$ and $1/L(ks,\chi)$, we conclude that the same argument applies and thus we have our lemma for $f=\mu^{(k)}\chi$.

        Again, noting that the factors $P(s)$ and $P(s)/P(s/k)$, that emerges in the Dirichlet series of $f=\mu^{(k)}g_\chi$, are negligible in the above estimates, the lemma for $f=\mu^{(k)}g_\chi$ follows readily.
    \end{proof}

For a general imaginary height $T$, we assume more than the simplicity of zeros and also must pay the price of an additional error term. We have
    \begin{lemma}\label{lemma-general-T}
        Assume the generalized Riemann hypothesis and conjectures from \Cref{first-hypothesis,second-hypothesis}. Let $\chi$ be a primitive real non-principal Dirichlet character modulo $q$ and $k\geq 2$ a fixed integer. Then, for any $T\geq 2$, $x\geq 2$,
            \begin{align*}
                \sum_{n\leq x}f(n)=\sum_{|\gamma|<T}\frac{L(\frac{\rho}{k},\chi)}{\rho}Z_f(\rho)x^{\rho/k}+\Tilde{E}(x,T),
            \end{align*}
        where $Z_f(s)$ is defined as in the previous lemma (\cref{def-Zf}), and
            \begin{align*}
                \Tilde{E}(x,T)\ll_\varepsilon 1+\frac{x\log x}{T}+\frac{x}{T^{1-\varepsilon}\log x}+x^{\varepsilon}T^{\varepsilon}+\frac{x^{\frac{1}{2k}}(\log T)^{\frac{1}{2}}}{T^{\varepsilon}},
            \end{align*}
        for any $\varepsilon>0$ small enough.
    \end{lemma}
    \begin{proof}
        Let $n\leq T\leq n+1$ and take $T_n$ such that $n\leq T_n\leq T\leq n+1$. Again, suppose $k$ is even and $f=\mu^{(k)}\chi$. By the previous lemma,
            \begin{align*}
                \sum_{n\leq x}f(n)=\bigg(\sum_{|\gamma|<T}-\sum_{T_n\leq |\gamma|\leq T}\bigg)\frac{L(\frac{\rho}{k},\chi)P(\rho)x^{\rho/k}}{\rho\zeta'(\rho)}+E(x,T),
            \end{align*}
        where $E(x,T)\ll_\varepsilon 1+\frac{x\log x}{T}+\frac{x}{T^{1-\varepsilon}\log x}+x^{\varepsilon}T^{\varepsilon}$.

        Applying Cauchy-Schwarz inequality, the second sum is bounded as follows:
            \begin{align*}
                \bigg|\sum_{T_n\leq \gamma\leq T}\frac{L(\frac{\rho}{k},\chi)P(\rho)x^{\rho/k}}{\rho\zeta'(\rho)}\bigg|&\ll x^{\frac{1}{2k}}\bigg(\sum_{T_n\leq \gamma\leq T}\frac{1}{|\zeta'(\rho)|^2} \bigg)^{\frac{1}{2}}\bigg(\sum_{T_n\leq \gamma\leq T}\bigg|\frac{L(\frac{\rho}{k},\chi)}{\rho}\bigg|^2 \bigg)^{\frac{1}{2}}\\
                &\ll_\varepsilon x^{\frac{1}{2k}}[T^{1+\frac{1}{k}-4\varepsilon}]^{\frac{1}{2}}\bigg(\sum_{T_n\leq \gamma\leq T}\bigg|\frac{1}{\gamma^{\frac{1}{2}+\frac{1}{2k}-\varepsilon}}\bigg|^2 \bigg)^{\frac{1}{2}}\\
                &\ll_\varepsilon x^{\frac{1}{2k}}T^{\frac{1}{2}+\frac{1}{2k}-2\varepsilon}\frac{(\log T)^{\frac{1}{2}}}{T^{\frac{1}{2}+\frac{1}{2k}-\varepsilon}}\ll_\varepsilon \frac{x^{\frac{1}{2k}}(\log T)^{\frac{1}{2}}}{T^{\varepsilon}}.
            \end{align*}

        Thus,
            \begin{align*}
                \Tilde{E}(x,T)\ll_\varepsilon E(x,T)+\frac{x^{\frac{1}{2k}}(\log T)^{\frac{1}{2}}}{T^{\varepsilon}}.
            \end{align*}

        The case where $k$ is odd follows similarly, with $L'(\rho,\chi)$ instead of $\zeta'(\rho)$ and without $P(\rho)$.

        If $f=\mu^{(k)}g_\chi$, the same reasoning as in the proof of \Cref{lemma-specific-T} applies: for $\Re(s)>0$, $P(s)$ and $1/P(ks)$ are $\ll 1$. Therefore, when applying Cauchy-Schwarz, the contribution is an implicit constant.
    \end{proof}

    \section{Proof of \texorpdfstring{\Cref{log-meas-zero-thm,cor-distribution}}{}}\label{proofs-section}

    Now we use Ng's result (\Cref{distribution-result-Ng}) to prove our theorem on the logarithmic distribution of $x^{-\frac{1}{2k}}\sum_{n\leq x}f(n)$.

    \begin{proof}[Proof of \Cref{cor-distribution}]

        First we observe that, as in the proof of \Cref{main-lemma}, taking the sequences $\lambda_n=\frac{\gamma_n}{k}$ and $r_n=\frac{L(\rho_n/k,\chi)P(\rho_n)}{\rho_n\zeta'(\rho_n)}$ (for even $k$), it follows that they satisfy \Cref{assumption-1,assumption-3}. Note that \Cref{assumption-2} for the zeros of the Riemann zeta function is the known result $N(T)\ll T\log T$ presented in \Cref{section-preliminaries}.

        Now, by \Cref{lemma-general-T},
            \begin{align*}
                \sum_{n\leq x}f(n)=\sum_{|\gamma|<T}\frac{L(\frac{\rho}{k},\chi)P(\rho)x^{\rho/k}}{\rho\zeta'(\rho)}+\sum_{T\leq |\gamma|<X}\frac{L(\frac{\rho}{k},\chi)P(\rho)x^{\rho/k}}{\rho\zeta'(\rho)}+\Tilde{E}(x,X).
            \end{align*}
            
        Taking $x=e^y$, $Y=\log X$ and dividing both sides by $e^{\frac{y}{2k}}$,
            \begin{align*}
                \frac{\sum_{n\leq e^y}f(n)}{e^{\frac{y}{2k}}}&=\bigg(\sum_{|\gamma|<T}+\sum_{T\leq |\gamma|<e^{Y}}\bigg)\frac{L(\frac{\rho}{k},\chi)P(\rho)e^{iy\gamma/k}}{\rho\zeta'(\rho)}+\frac{\Tilde{E}(e^y,e^{Y})}{e^{\frac{y}{2k}}}.
            \end{align*}

        Next, by \Cref{main-lemma-corollary}, we have for the middle term:
            \begin{align*}
                \int_{y_0}^Y\bigg|\sum_{T\leq |\gamma|<e^{Y}}&\frac{L(\frac{\rho}{k},\chi)P(\rho)e^{iy\gamma/k}}{\rho\zeta'(\rho)}\bigg|^2\,dy\\
                &\ll\sum_{j=0}^{[Y]-y_0-1}\int_{y_0+j}^{y_0+j+1}\bigg|\sum_{T\leq |\gamma|<e^{Y}}\frac{L(\frac{\rho}{k},\chi)P(\rho)e^{iy\gamma/k}}{\rho\zeta'(\rho)}\bigg|^2\,dy\\
                &\ll_\varepsilon\sum_{j=0}^{[Y]}\frac{1}{T^{\varepsilon}}\ll_\varepsilon \frac{Y}{T^{\varepsilon}}.
            \end{align*}

        From the error term in \Cref{lemma-general-T}, choosing $\varepsilon<\frac{1}{4k}$, it follows that
            \begin{align*}
                \int_{y_0}^Y\bigg|&\frac{\Tilde{E}(e^y,e^{Y})}{e^{\frac{y}{2k}}}\bigg|^2\,dy\\
                &\ll_\varepsilon \int_{y_0}^Y \frac{1}{e^{\frac{y}{k}}}+\frac{e^{y(2-\frac{1}{k})}y^2}{e^{2Y}}+\frac{e^{y(2-\frac{1}{k})}}{e^{2Y(1-\varepsilon)}y^2}+e^{y(2\varepsilon-\frac{1}{k})}e^{2Y\varepsilon}+\frac{Y}{e^{2Y\varepsilon}}\,dy\\
                &\ll_\varepsilon 1.
            \end{align*}

        Putting these two results together,
            \begin{align*}
                \lim_{Y\to\infty}\frac{1}{Y}\int_{y_0}^Y\bigg|\sum_{T\leq |\gamma|<e^{Y}}\frac{L(\frac{\rho}{k},\chi)P(\rho)e^{iy\gamma/k}}{\rho\zeta'(\rho)}+\frac{\Tilde{E}(e^y,e^{Y})}{e^{\frac{y}{2k}}}\bigg|^2\,dy\ll_\varepsilon \frac{1}{T^{\varepsilon}}.
            \end{align*}

        Finally, since the zeros of the Riemann zeta function are reflected on the real axis and $z+\overline{z}=2\Re(z)$, if we label the ordinates of those zeros up to height $T$ as we did before, e.g., $\rho_n=\frac{1}{2}+i\gamma_n$, we have
            \begin{align*}
                \sum_{|\gamma|<T}\frac{L(\frac{\rho}{k},\chi)P(\rho)e^{iy\gamma/k}}{\rho\zeta'(\rho)}=2\Re\bigg(\sum_{n=1}^{N(T)} \frac{L(\frac{\rho_n}{k},\chi)P(\rho_n)e^{iy\gamma_n/k}}{\rho_n\zeta'(\rho_n)}\bigg).
            \end{align*}

        Therefore, choosing
            \begin{align*}
                \varphi(y)=\frac{\sum_{n\leq e^y}f(n)}{e^{\frac{y}{2k}}},
            \end{align*}
        the existence of the limiting distribution follows from \Cref{distribution-result-Ng}.

        The same proof goes for odd $k$. \Cref{assumption-1,assumption-3} are satisfied (also as in the proof of \Cref{main-lemma}, since it treats the more general case where the character is not necessarily real) and \Cref{assumption-2} is \Cref{classical-zeros-bound}.

        For $f=\mu^{(k)}g_\chi$, by the same arguments we made in the preceding lemmas regarding the contribution of $P(s)$ and $1/P(ks)$, we obtain our result.
    \end{proof}

Next we prove \Cref{log-meas-zero-thm}. As we remarked earlier, this result essentially shows, under stronger conditions, that the bound $\ll x^{\frac{1}{2k}+\varepsilon}$ for the partial sums of $f=\mu^{(k)}\chi$ (resp. $\mu^{(k)}g_\chi$) is attainable.
    \begin{proof}[Proof of \Cref{log-meas-zero-thm}]
        Assume that $k$ is even (the odd case follows from the exact same proof, so we omit the details). Starting with \Cref{lemma-general-T}, we have
            \begin{align}\label{split-sum-thm2}
                \sum_{n\leq x}f(n)&=\bigg(\sum_{|\gamma|<\log T}+\sum_{\log T\leq |\gamma| <T}\bigg)\frac{L(\frac{\rho}{k},\chi)P(\rho)x^{\rho/k}}{\rho\zeta'(\rho)}\\
                &+O_\varepsilon\bigg(1+\frac{x\log x}{T}+\frac{x}{T^{1-\varepsilon}\log x}+x^{\varepsilon}T^{\varepsilon}+\frac{x^{\frac{1}{2k}}(\log T)^{\frac{1}{2}}}{T^{\varepsilon}}\bigg)\nonumber,
            \end{align}
        for any small enough $\varepsilon>0$.

        First note that, if $T\leq x\ll T$, the error term is $O_\varepsilon(x^{\frac{1}{2k}-\varepsilon}(\log x)^{\frac{1}{2}})$.

        By Cauchy-Schwarz inequality, the bound $\sum_{0<\gamma\leq T}|\zeta'(\rho)|^{-2}\ll T^{1+\frac{1}{k}-\varepsilon}$ implies
            \begin{align*}
                \sum_{0<\gamma\leq T}\frac{1}{|\zeta'(\rho)|}&\leq \bigg(\sum_{0<\gamma\leq T}\frac{1}{|\zeta'(\rho)|^2} \bigg)^{\frac{1}{2}}\bigg(\sum_{0<\gamma\leq T}1\bigg)^{\frac{1}{2}}\\
                &\ll T^{\frac{1}{2}+\frac{1}{2k}-\frac{\varepsilon}{2}}T^{\frac{1}{2}}(\log T)^{\frac{1}{2}}\\
                &\ll T^{1+\frac{1}{2k}-\frac{\varepsilon}{2}}(\log T)^{\frac{1}{2}}.
            \end{align*}

        Now, for the smaller sum on the right hand side of \eqref{split-sum-thm2}, we use the above estimate and \cref{gamma-factor-bound}. We have,
            \begin{align*}
                \bigg|\sum_{|\gamma|<\log T}\frac{L(\frac{\rho}{k},\chi)P(\rho)x^{\rho/k}}{\rho\zeta'(\rho)}\bigg|&\ll x^{\frac{1}{2k}}\sum_{0<\gamma <\log T}\bigg|\frac{L(\frac{\rho}{k},\chi)P(\rho)}{\rho\zeta'(\rho)}\bigg|\\
                & \ll x^{\frac{1}{2k}}\frac{(\log T)^{\frac{1}{2}-\frac{1}{2k}}}{\log T}\sum_{0<\gamma<\log T}\frac{1}{|\zeta'(\rho)|}\\
                &\ll x^{\frac{1}{2k}}\frac{(\log T)^{\frac{1}{2}-\frac{1}{2k}}}{\log T}(\log T)^{1+\frac{1}{2k}-\frac{\varepsilon}{2}}(\log\log T)^{\frac{1}{2}}\\
                &\ll_\varepsilon x^{\frac{1}{2k}}(\log T)^{\frac{1}{2}+\varepsilon}.
            \end{align*}
        
        Restricting ourselves to $x\in [T,eT]$, let $C_1$ be the constant that is implicit in the above estimate and $C_2$ the one coming from the error term of \eqref{split-sum-thm2}. We thus have
            \begin{align}\label{implicit-cte-eq}
                \bigg|\sum_{\log T\leq |\gamma|<T}\frac{L(\frac{\rho}{k},\chi)P(\rho)x^{\rho/k}}{\rho\zeta'(\rho)}\bigg|\nonumber&\geq \bigg|\sum_{n\leq x}f(n)\bigg|-C_1x^{\frac{1}{2k}}(\log x)^{\frac{1}{2}+\varepsilon}-C_2x^{\frac{1}{2k}-\varepsilon}(\log x)^{\frac{1}{2}}\\
                &=\bigg|\sum_{n\leq x}f(n)\bigg|-\bigg(C_1+\frac{C_2}{x^\varepsilon(\log x)^\varepsilon}\bigg)x^{\frac{1}{2k}}(\log x)^{\frac{1}{2}+\varepsilon}\nonumber\\
                &\geq\bigg|\sum_{n\leq x}f(n)\bigg|-(C_1+C_2)x^{\frac{1}{2k}}(\log x)^{\frac{1}{2}+\varepsilon}.
            \end{align}

        Define the set $S=\{x\geq e^2 : |\sum_{n\leq x}f(n)| \geq \Tilde{C} x^{\frac{1}{2k}}(\log x)^{\frac{1}{2}+\varepsilon} \}$, where $\Tilde{C}$ is a sufficiently large constant to be chosen later. If $x\in S\cap[T,eT]$, it follows from \cref{implicit-cte-eq} that
            \begin{align*}
                \bigg|\sum_{\log T\leq |\gamma|<T}\frac{L(\frac{\rho}{k},\chi)P(\rho)x^{\rho/k}}{\rho\zeta'(\rho)}\bigg|\geq (\Tilde{C}-(C_1+C_2))x^{\frac{1}{2k}}(\log x)^{\frac{1}{2}+\varepsilon}.
            \end{align*}
        If $\Tilde{C}>C_1+C_2+1$, the inequality above is valid even if we replace the constant by $1$.

        Taking $T=e^M$, then $x\in S\cap[e^M,e^{M+1}]$ and, by squaring both sides, dividing through by $x^{1+\frac{1}{k}}$ and integrating, we obtain
            \begin{align*}
                M^{1+2\varepsilon}\int_{S\cap[e^M,e^{M+1}]}\,\frac{dx}{x}&\leq \int_{e^M}^{e^{M+1}}\bigg|\sum_{M\leq |\gamma|<e^M}\frac{L(\frac{\rho}{k},\chi)P(\rho)x^{\rho/k}}{\rho\zeta'(\rho)}\bigg|^2\,\frac{dx}{x^{1+\frac{1}{k}}}\ll_{k,\varepsilon} \frac{1}{M^{\varepsilon}},
            \end{align*}
        where the last inequality follows from \Cref{main-lemma-corollary}. Therefore,
            \begin{align*}
                \int_{S\cap [1,X]}\,\frac{dx}{x}&\leq\sum_{M=2}^{[\log X]}\int_{S\cap [e^M,e^{M+1}]}\,\frac{dx}{x}\ll_{k,\varepsilon} \sum_{M=2}^{[\log X]}\frac{1}{M^{1+3\varepsilon}}\ll_{k,\varepsilon} 1.
            \end{align*}

        The above shows that the set $S$ has finite logarithmic measure, so we obtain our theorem.
    \end{proof}

\section{Large deviations}\label{large-deviation-section}

In this section we assume\footnote{The same justifications from previous sections can be made to convince the reader that the results in this section also holds for $f=\mu^{(k)}g_\chi$. To avoid excessive repetition, here we won't expand for the modified character.} $f=\mu^{(k)}\chi$ and the following conjectures:
    \begin{conjecture}[Linear Independence Conjecture - LI]
        The set $\{\gamma \in\mathbb{R}_{>0} : \zeta(\beta+i\gamma)=0, \beta\in(0,1)\}$ is linearly independent over $\mathbb{Q}$.
    \end{conjecture}
    \begin{conjecture}[Generalized Linear Independence Conjecture - GLI]
        For any fixed $q\geq 3$ and $\kappa\geq 1$ integer, the set $\cup_{\chi^\kappa (q)}\{\gamma\in\mathbb{R}_{>0} : L(\beta+i\gamma,\chi^\kappa)=0, \beta\in(0,1)\}$ is linearly independent over $\mathbb{Q}$. 
    \end{conjecture}
That is to say, if we take $\gamma_1,\dots,\gamma_n$ positive imaginary parts of non-trivial zeros of the Riemann zeta function (resp. Dirichlet $L$-function), we have that
    \begin{align*}
        a_1\gamma_1+\cdots+a_n\gamma_n =0
    \end{align*}
only if $a_i=0,\forall i$.

Let $V\in\mathbb{R}$ be fixed and define
    \begin{align*}
        g(x)=   \begin{cases}
                    1   &\text{ if $x\geq V$,}\\
                    0   &\text{ otherwise.}
                \end{cases}
    \end{align*}

From \Cref{def-limiting-distribution}, in the case where $\nu_k$ is absolutely continuous, our \Cref{cor-distribution} implies
    \begin{align*}
        \lim_{Y\to \infty}\frac{1}{Y}\meas\bigg\{y\in[0,Y] : \frac{\sum_{n\leq e^y}f(n)}{e^{\frac{y}{2k}}}\geq V\bigg\}=\int_{V}^\infty d\nu_k(x)=\nu_k([V,\infty)),
    \end{align*}
where $\meas\{\cdot\}$ is the Lebesgue measure on $\mathbb{R}$. Roughly, we would like to study how often $\sum_{n\leq e^{y}}f(n)$ takes values greater than $e^{\frac{y}{2k}}V$.

Therefore, defining a proper random variable $X(\boldsymbol{\theta})$ (that comes from the construction of the limiting distribution $\nu_k$ above), we would like to estimate bounds for
    \begin{align*}
        \nu_k([V,\infty))=\mathbb{P}(X(\boldsymbol{\theta})\geq V).
    \end{align*}

Turns out this is identical to the upper and lower bounds obtained by Meng for the distribution of $k$-free numbers, so we just highlight here the results and main concepts of the proof.

The following proposition shows that the Fourier transform of $\nu_k$ can be explicitly represented, as long as we assume the LI/GLI conjecture.
    \begin{proposition}[\citeauthor{Ng2004,Meng2017}]
        Assume the generalized Riemann hypothesis. Additionally, assume \Cref{first-hypothesis} and LI conjecture or \Cref{second-hypothesis} and GLI conjecture. Then, the Fourier transform $\hat{\nu}_k(\xi)=\int_{\mathbb{R}} e^{-i\xi t}\,d\nu_k(t)$ is equal to
            \begin{align*}
                \hat{\nu}_k(\xi)=\prod_{\gamma>0} \mathcal{J}_0\bigg(2\xi\bigg|\frac{L(\frac{\rho}{k},\chi)P(\rho)}{\rho \zeta'(\rho)}\bigg|\bigg) \quad\text{ or }\quad \hat{\nu}_k(\xi)=\prod_{\gamma>0} \mathcal{J}_0\bigg(2\xi\bigg|\frac{L(\frac{\rho}{k},\chi)}{\rho L'(\rho,\chi)}\bigg|\bigg),
            \end{align*}
        if $k$ is even or odd, respectively. Here $\mathcal{J}_0(z)=\sum_{m=0}^\infty\frac{(-1)^m(z/2)^{2m}}{(m!)^2}$, is the Bessel function of order $0$.
    \end{proposition}
The proof follows similarly to Corollary 1 of \cite{Ng2004} or Theorem 6.1 of \cite{Humphries2013} and the tool that allows us to compute this Fourier transform is the construction of $\nu_k$ by means of the Kronecker-Weyl Theorem \cite[See][Lemma~5.3]{Humphries2013}.

Now, as in the work of \citeauthor{Ng2004} and \citeauthor{Meng2017}, we let $\boldsymbol{\theta}=(\theta_1,\theta_2,\dots)\in\mathbb{T}^\infty$. Assuming the LI/GLI conjectures, the limiting distribution $\nu_k$ is equal to $\nu_{k,X}(x)=\mathbb{P}(X^{-1}(-\infty,x))$, where $X$ is the random variable on the infinite torus defined by
    \begin{align}\label{rv-even}
        X(\boldsymbol{\theta})=
                \begin{dcases}
                    2\sum_{\gamma>0}\bigg|\frac{L(\frac{\rho}{k},\chi)P(\rho)}{\rho\zeta'(\rho)}\bigg|\sin(2\pi \theta_\gamma) &\text{ if $k$ is even,}\\
                    2\sum_{\gamma>0}\bigg|\frac{L(\frac{\rho}{k},\chi)}{\rho L'(\rho,\chi)}\bigg|\sin(2\pi\theta_\gamma)   &\text{ if $k$ is odd.}
                \end{dcases}
    \end{align}

Thus, under these assumptions, one can study $\nu_k$ by means of the above random variable $X(\boldsymbol{\theta})$ defined on the infinite torus $\mathbb{T}^\infty$.

The main tool employed by previous authors to work with sine random variables as in \eqref{rv-even} is the following result of \citeauthor{Montgomery1980}:
    \begin{lemma}\label{Montgomery-prob-lemma}
        Let $X(\boldsymbol{\theta})=\sum_{l=1}^\infty r_l\sin(2\pi \theta_l)$, where $\sum_{l=1}^\infty r_l^2<\infty$. For any $K\geq 1$ integer,
            \begin{align*}
                \mathbb{P}\bigg(X(\boldsymbol{\theta})\geq 2\sum_{l=1}^K r_l\bigg)\leq \exp\bigg(-\frac{3}{4}\bigg(\sum_{l=1}^K r_l\bigg)^2\bigg(\sum_{l=K+1}^\infty r_l^2\bigg)^{-1}\bigg),
            \end{align*}
        and
            \begin{align*}
                \mathbb{P}\bigg(X(\boldsymbol{\theta})\geq \frac{1}{2}\sum_{l=1}^K r_l\bigg)\geq \frac{1}{2^{40}}\exp\bigg(-100\bigg(\sum_{l=1}^K r_l\bigg)^2\bigg(\sum_{l=K+1}^\infty r_l^2\bigg)^{-1}\bigg).
            \end{align*}
    \end{lemma}

By letting $r_\gamma=2|\frac{L(\rho/k,\chi)P(\rho)}{\rho\zeta'(\rho)}|$ (resp. $r_\gamma=2|\frac{L(\rho/k,\chi)}{\rho L'(\rho,\chi)}|$), similar to Sections 4 and 5 of \cite{Meng2017}, we obtain
    \begin{theorem}\label{large-dev-theorem}
        Let $k\geq 2$ be a fixed integer, assume the generalized Riemann hypothesis and also one of the following:
            \begin{enumerate}[(i)]
                \item $k$ even, LI conjecture and $\sum_{0<\gamma\leq T}|\zeta'(\rho)|^{-2}\ll T^{1+\varepsilon},\forall\varepsilon>0$;
                \item $k$ odd, GLI conjecture and $\sum_{0<\gamma\leq T}|L'(\rho,\chi)|^{-2}\ll T^{1+\varepsilon},\forall\varepsilon>0$.
            \end{enumerate}
        Then, there exists constants $c_1,c_2>0$ such that, for any $\varepsilon>0$,
            \begin{align}\label{large-dev-part1}
                \exp(-c_1V^{\frac{2k}{k-1}+\varepsilon})\leq \nu_k([V,\infty))\leq \exp(-c_2 V^{\frac{2k}{k-1}-\varepsilon}).
            \end{align}

        Moreover, if we respectively assume
            \[
                \sum_{0<\gamma\leq T}\frac{1}{|\zeta'(\rho)|^{2l}}\asymp T(\log T)^{(l-1)^2}\quad\text{ or }\quad \sum_{0<\gamma\leq T}\frac{1}{|L'(\rho,\chi)|^{2l}}\asymp T(\log T)^{(l-1)^2},
            \]
        for $l<\frac{3}{2}$ (see \Cref{Gonek-Hejhal-conjecture} and \cref{conjecture-L}), then there exists a constant $c_3>0$ such that
            \begin{align}\label{large-dev-part2}
                \nu_k([V,\infty))\ll \exp\bigg(-c_3\frac{V^{\frac{2k}{k-1}}}{(\log V)^{\frac{1}{2(k-1)}+o(1)}}\bigg).
            \end{align}
    \end{theorem}
    
The more cautious analysis that leads to \cref{large-dev-part2} gives support to a conjecture that equates the Large Deviation Conjecture in \cite{Meng2017}, i.e., the right hand side gives the precise order of $\nu_k([V,\infty))$. Additionally, as a consequence of this and the analysis of Ng \cite[See][Section~4.3]{Ng2004}, we have the following:
    \begin{conjecture}\label{conjecture-chi}
        Let $k\geq 2$ be an integer, $\chi$ a real non-principal Dirichlet character modulo $q$ and $f=\mu^{(k)}\chi$ (resp. $\mu^{(k)}g_\chi$). Then, there exists a constant $C_k>0$, such that
            \begin{align*}
                \underline{\overline{\lim}}_{x\to\infty}\frac{\sum_{n\leq x}f(n)}{x^{\frac{1}{2k}}(\log\log x)^{\frac{1}{2}-\frac{1}{2k}}(\log\log\log x)^{\frac{1}{4k}}}=\pm C_k.
            \end{align*}
    \end{conjecture}

\subsection{\texorpdfstring{Main ideas of the proof of \Cref{large-dev-theorem}}{}}

To prove the first statement \eqref{large-dev-part1}, considering \eqref{rv-even}, we take $r_\gamma=2|\frac{L(\rho/k,\chi)P(\rho)}{\rho\zeta'(\rho)}|$ and our aim is to apply Montgomery's Lemma.

From \cref{gamma-factor-bound} and from \Cref{Lemma-lowerupper-bound-zeta} combined with Cauchy's Integral Formula, we have that
    \[
        \bigg|L\bigg(\frac{\rho}{k},\chi\bigg)\bigg|\gg |\gamma|^{\frac{1}{2}-\frac{1}{2k}-\varepsilon}\quad\text{ and }\quad\frac{1}{|\zeta'(\rho)|}\gg \frac{1}{|\gamma|^\varepsilon},
    \]
respectively.

The hypothesis of the theorem and the above estimates gives the upper and lower bounds:
    \[
        T^{\frac{1}{2}-\frac{1}{2k}-\varepsilon}\ll \sum_{0<\gamma\leq T}r_\gamma\ll T^{\frac{1}{2}-\frac{1}{2k}+\varepsilon}\quad\text{ and }\quad \frac{1}{T^{\frac{1}{k}+\varepsilon}}\ll\sum_{\gamma>T}r_\gamma^2\ll \frac{1}{T^{\frac{1}{k}-\varepsilon}}.
    \]

Now we apply \Cref{Montgomery-prob-lemma} to obtain
    \[
        \exp(-c_1T^{1+\varepsilon})\leq \nu_k\bigg(\bigg[2\sum_{0<\gamma \leq T}r_\gamma,\infty\bigg)\bigg)\leq \exp(-c_2T^{1-\varepsilon}),
    \]
for constants $c_1,c_2>0$ and conclude by taking $V=T^{\frac{1}{2}-\frac{1}{2k}}$.

To prove the second part of \Cref{large-dev-theorem}, namely \cref{large-dev-part2}, we need the following result from \cite{Meng2017}: under the Riemann hypothesis, for $w\in(0,1)$ and $r\geq 1$, we have
    \[
        \sum_{0<\gamma\leq T}|\zeta(1-w\rho)|^{2r}\ll T\log T,
    \]
and
    \[
        \sum_{0<\gamma\leq T}\frac{1}{|\zeta(1-w\rho)|^{2r}}\ll T\log T.
    \]

With the above at our disposal, we use Hölder's inequality and the conjectural bounds for the $2l$-th moments to estimate the sums over $r_\gamma$ and $r_\gamma^2$. By following Meng's steps, we are able to produce the extra logarithmic term in the denominator of \eqref{large-dev-part2}.

\section{Additional results}\label{section-additional-results}

This section is dedicated to present theorems that are somewhat similar to the Weak Mertens Hypothesis, which asserts that,
    \begin{align}\label{weak-mertens-conj}
        \int_1^X\bigg(\frac{\sum_{n\leq x}\mu(n)}{x}\bigg)^2\,dx\ll \log X.
    \end{align}

The above is a notably strong conjecture and, in particular, \cref{weak-mertens-conj} implies the Riemann hypothesis, that all the zeros of $\zeta$ are simple, the bound $\zeta'(\rho)^{-1}\ll |\rho|$ and the convergence of the sum $\sum_\gamma |\rho\zeta'(\rho)|^{-2}$ \cite[See][Section~14.29]{TitchmarshBookRZF}.

Our result states:
    \begin{theorem}\label{weak-mertens-analogue}
        Assume the generalized Riemann hypothesis and \Cref{first-hypothesis} or \Cref{second-hypothesis}. Let $\chi$ be a primitive real non-principal Dirichlet character of modulus $q$ and $k\geq 2$ a fixed integer. Then,
            \begin{align*}
                \int_2^X\bigg(\frac{\sum_{n\leq x} f(n)}{x^{\frac{1}{2k}}}\bigg)^2\,\frac{dx}{x}\ll \log X.
            \end{align*}
    \end{theorem}
    \begin{proof}            
        Let $k$ be even. Similar to the proof of \Cref{log-meas-zero-thm}, we let $x\in [X,eX]$ and apply \Cref{lemma-general-T} to obtain:
            \begin{align*}
                \bigg(\sum_{n\leq x}f(n)\bigg)^2\ll\bigg|\sum_{|\gamma|< X}\frac{L(\frac{\rho}{k},\chi)P(\rho)x^{\rho/k}}{\rho\zeta'(\rho)}\bigg|^2+O_\varepsilon(X^{\frac{1}{k}-2\varepsilon}(\log X)).
            \end{align*}

        If we take $X=e^M$, then $x\in[e^M,e^{M+1}]$. Thus, by \Cref{main-lemma-corollary},
            \begin{align*}
                \int_{e^M}^{e^{M+1}}&\frac{(\sum_{n\leq x}f(n))^2}{x^{1+\frac{1}{k}}}\,dx\\
                &\ll_\varepsilon \int_{e^M}^{e^{M+1}}\bigg|\sum_{14<|\gamma|< e^M}\frac{L(\frac{\rho}{k},\chi)P(\rho)x^{\rho/k}}{\rho\zeta'(\rho)}\bigg|^2\,\frac{dx}{x^{1+\frac{1}{k}}}+O(e^{-2M\varepsilon}M)\\
                &\ll_\varepsilon 1.
            \end{align*}

        Finally, by the above and summing over $M$ in the same fashion of \Cref{log-meas-zero-thm}, we have
            \begin{align*}
                \int_2^X\bigg(\frac{\sum_{n\leq x}f(n)}{x^{\frac{1}{2k}}}\bigg)^2\,\frac{dx}{x}\ll \sum_{M=2}^{[\log X]}\int_{e^M}^{e^{M+1}}\bigg(\frac{\sum_{n\leq x}f(n)}{x^{\frac{1}{2k}}}\bigg)^2\,\frac{dx}{x}\ll_\varepsilon \log X.
            \end{align*}
    \end{proof}

    \begin{remark}
        Using an argument similar to Theorem 14.29 (B) of \cite{TitchmarshBookRZF}, one could also show that, under the generalized Riemann hypothesis, the bound
            \[
                \int_2^X\bigg(\frac{\sum_{n\leq x} f(n)}{x^{\frac{1}{2k}}}\bigg)^2\,\frac{dx}{x}\ll \log X
            \]
        implies that
            \[
                \sum_{\gamma}\bigg|\frac{L(\frac{\rho}{k},\chi)P(\rho)}{\rho\zeta'(\rho)}\bigg|^2
            \]
        is convergent and a further consequence of the latter is $\sum_{0<\gamma\leq T}|\zeta'(\rho)|^{-2}\ll T^{1+\frac{1}{k}+\varepsilon}$.
    \end{remark}

As a consequence of the existence of a limiting distribution, by Theorem 1.14 of \citeauthor{Akbary-Ng-Shahabi2014} (also see \cite[][Theorem~3]{Ng2004} and \cite[][Theorem~2]{Meng2017}), we obtain a stronger result, namely:
    \begin{align*}
        \int_{2}^X\bigg(\frac{\sum_{n\leq x}f(n)}{x^{\frac{1}{2k}}}\bigg)^2\,\frac{dx}{x}\sim \beta_k \log X,
    \end{align*}
where, for $f=\mu^{(k)}\chi$,
    \begin{align*}
        \beta_k=
            \begin{dcases}
                2\sum_{\gamma>0}\bigg|\frac{L(\frac{\rho}{k},\chi)P(\rho)}{\rho\zeta'(\rho)}\bigg|^2   &\text{ if $k$ is even,}\\
                2\sum_{\gamma>0}\bigg|\frac{L(\frac{\rho}{k},\chi)}{\rho L'(\rho,\chi)}\bigg|^2   &\text{ if $k$ is odd.}
            \end{dcases}
    \end{align*}
As usual, if $f=\mu^{(k)}g_\chi$, we replace the factor $P(\rho)$ by $P(\rho/k)$ if $k$ is even and include $P(\rho/k)/P(\rho)$ if $k$ is odd.

At last, note that $\sum_{0<\gamma\leq T}|\zeta'(\rho)|^{-2}\ll T^{1+\frac{1}{k}-\varepsilon}$ (resp. $\sum_{0<\gamma\leq T}|L'(\rho,\chi)|^{-2}\ll T^{1+\frac{1}{k}-\varepsilon}$) can be used to show the convergence of $\beta_k$ above.

\section{Final remarks}

Throughout this manuscript we assumed that the second negative moment for $\zeta(s)$ and $L(s,\chi)$ has cancellation $\ll T^{1+\frac{1}{k}-\varepsilon}$ and, for \Cref{log-meas-zero-thm}, we obtained a cancellation of $\ll T^{1+\frac{1}{2k}-\frac{\varepsilon}{2}}(\log T)^{\frac{1}{2}}$ for the first negative moment, as a consequence of the former. However, for some summatory functions, e.g. partial sums of the Möbius function, as pointed out by \citeauthor{Meng2017} (also covered in the recent work of \citeauthor{ng2025primenumbererrorterms}), one can assume the much weaker bound $\sum_{0<\gamma\leq T}|\zeta'(\rho)|^{-2}\ll T^{2-\varepsilon}$.

Assuming the Riemann hypothesis and the simplicity of zeros, we note that we have $\sum_{0<\gamma\leq T}|\zeta'(\rho)|^{-2}\ll T^{2+\varepsilon}$, since $|\zeta'(\rho)|^{-1}\ll |\rho|$. Furthermore, the Weak Mertens Hypothesis \eqref{weak-mertens-conj} implies that
    \[
        \sum_{0<\gamma\leq T}\frac{1}{|\zeta'(\rho)|^{2r}}=
            \begin{cases}
                o(T^{r+1}(\log T)^{1-r})    &\text{ if $0<r<1$,}\\
                o(T^{2r})   &\text{ if $r\geq 1$.}
            \end{cases}
    \]
This result follows from Hölder's inequality and the fact that $\sum_{\gamma}|\rho\zeta'(\rho)|^{-2}$ is convergent under this conjecture. We refer to \citeauthor{Bui-Florea-Milinovich} for a proof of the above and, although the authors considered the sum where $\gamma\in (T,2T]$, note that the same proof is carried for sums over all zeros of $\zeta$ with ordinates in $(0,T]$.

Also in \cite{Bui-Florea-Milinovich}, assuming only the Riemann hypothesis, \citeauthor{Bui-Florea-Milinovich} showed that
    \begin{align}\label{Bui-Florea-Milinovich-result}
        \sum_{\gamma\in \mathcal{F}}\frac{1}{|\zeta'(\rho)|^{2r}}\ll_{r,\delta}
            \begin{cases}
                T^{1+\delta}    &\text{ if $r<\frac{1}{2}$,}\\
                T^{r+\frac{1}{2}+\delta}    &\text{ if $r\geq \frac{1}{2}$,}
            \end{cases}
    \end{align}
for any $\delta>0$, where
    \[
        \mathcal{F}=\bigg\{\gamma\in (T,2T] : |\gamma-\gamma'|\gg \frac{1}{\log T} \text{ and $\gamma'$ is any other ordinate of a zero of $\zeta$}\bigg\}.
    \]

As remarked by the authors, the set $\mathcal{F}$ is expected to be arbitrarily close to full density inside the dyadic interval $(T,2T]$. Thus, if such result could be extended to sums over all ordinates $\gamma\in (0,T]$, only the generalized Riemann hypothesis would be enough to prove our results, since \Cref{first-hypothesis,second-hypothesis} from \Cref{cor-distribution} would become obsolete.

In prior works related to limiting distribution of summatory functions, the strategy for establishing explicit formulas, as in \Cref{lemma-specific-T}, was to use Cauchy's Integral Formula and the Residue Theorem in a rectangle of vertices $c\pm iT$ and $-U\pm iT$, for some $c,U>0$, estimate the integrals over each line segment and take $U\to \infty$. Our case is more delicate due to the function $P(s)$. If we go beyond the imaginary axis, the contribution coming from the poles described in \cref{zeros-P} would be too large, as discussed in \Cref{subsection-discussion-P(s)}. This explains why we maintained our contour in the half-plane $\Re(s)>0$.

The usual approach described above works when we consider the case where $k$ is odd and we are working with the function $\mu^{(k)}\chi$, since the representation of its Dirichlet series does not contain the factor $P(s)$, as presented in \Cref{section-preliminaries}.

\section{Acknowledgements}

I would like to thank Marco Aymone for first pointing out the problem, his support and for reading an early version. I am also thankful for Professor Winston Heap for his interest on the topic and also suggesting the problem, and helpful discussions while I was visiting the Max Planck Institute for Mathematics. Moreover, I thank the MPIM for the warm hospitality. Finally, I would like to thank Professor Micah Milinovich for providing a reference that contributed to the improvement of the exposition.

\begin{Backmatter}

\bibliographystyle{apalike}
\bibliography{Sample-refs.bib}

\printaddress

\end{Backmatter}

\end{document}